\documentclass[10pt]{article}
\usepackage[italian]{babel}
\usepackage[latin1]{inputenc}
\usepackage[T1]{fontenc}
\usepackage{amsmath}
\usepackage{amssymb}
\usepackage{amsfonts}
\usepackage{amsthm}
\usepackage{mathrsfs}
\usepackage{graphicx}
\usepackage{makeidx}
\usepackage{hyperref}
\hypersetup{ colorlinks=true,linkcolor=blue, citecolor=blue, urlcolor=blue  }

\addto\captionsitalian{}
\addto\captionsitalian{}

\theoremstyle{plain}
\newtheoremstyle{myexstyle}
{}{}{}{}{\bfseries}{.}{ }{}
\newtheorem{definition}{\textsc{Definizione}}[section]

\newtheorem{theorem}{\textsc{Teorema}}[section]
\newtheorem{proposition}{\textsc{Proposizione}}[section]

\theoremstyle{myexstyle}
\newtheorem{remark}[theorem]{\textsc{Osservazione}}
\newtheorem{example}{\textsc{Esempio}}[section]

\title{Varietà di Jacobi parziali}
\author{Patrick Cabau\thanks{\footnotesize \texttt{patrickcabau@gmail.com}} \\ {\footnotesize https://orcid.org/0000-0003-1861-6180} }
\date{}

\begin{document}

\maketitle
\renewcommand{\abstractname}{Riassunto}
\begin{abstract}
Introduciamo la nozione di varietà di Jacobi parziale nel quadro conveniente ($c^\infty$-completo) di Frölicher, Kriegl e Michor. Forniamo esempi espliciti sia in dimensione finita sia in dimensione infinita e analizziamo la distribuzione caratteristica associata a tale struttura. Concludiamo segnalando alcune direzioni di ricerca che potrebbero essere approfondite in studi futuri.
\end{abstract}

\newenvironment{englishabstract}{
    \renewcommand{\abstractname}{Abstract}%
    \begin{abstract}
}{
    \end{abstract}
}
\begin{englishabstract}
The notion of partial Jacobi manifold is introduced in the convenient ($c^\infty$-complete) framework of Frölicher, Kriegl, and Michor. Explicit examples are provided in both finite and infinite dimensions, and the characteristic distribution associated with this structure is analysed. Several research directions that would merit further study are indicated.
\end{englishabstract}

% macros	
\newcommand{\N}{\mathbb{N}}
\newcommand{\R}{\mathbb{R}}
\newcommand{\C}{\mathbb{C}}
\newcommand{\e}{\text{e}}
\newcommand{\pr}{\operatorname{pr}}


%\tableofcontents

\medskip
\noindent
\textbf{Ringraziamenti.} 
L'autore ringrazia sinceramente il Prof. Fernand Pelletier per le stimolanti discussioni sulla nozione di struttura parziale, la quale trova applicazione in diversi ambiti di ricerca.

\section{Introduzione}
\label{__Introduzione}

Le varietà di Jacobi di dimensione finita sono state introdotte in modo indipendente da A. Kirillov in \cite{Kir76} e A. Lichnerowicz in \cite{Lic78} \textit{via} definizioni diverse ma equivalenti (cf. \cite{Marl91}).\\

Queste strutture, che generalizzano simultaneamente le strutture di Poisson, quelle simplettiche e quelle di contatto, permettono in particolare di modellizzare sistemi dissipativi\footnote{Per sistemi dissipativi la cui dinamica è descritta mediante la parentesi di Jacobi, la perdita di energia può essere codificata tramite il campo vettoriale $E$.} (cf. \cite{DeLVa19}), analogamente a quanto avviene per i sistemi metriplettici\footnote{Un \emph{sistema metriplettico}\index{sistema!metriplettico} è costituito da una varietà differenziabile $M$ di dimensione finita, due applicazioni lisce $P$ e $G$ dal fibrato cotangente $T^\ast M$ al fibrato tangente $TM$ sopra l'identità, e due funzioni lisce:
\begin{enumerate}
\item[$\bullet$] 
l'Hamiltoniana $H$ o energia totale del sistema;
\item[$\bullet$] 
l'entropia $S$ del sistema
\end{enumerate}
tali che
\begin{description}
\item[\textbf{(SMF1)}] 
$\{f,g\} := \langle df,P(dg) \rangle$ è una parentesi di Poisson;
\item[\textbf{(SMF2)}] 
$(f,g) := \langle df,G(dg) \rangle$ è una  parentesi simmetrica semidefinita positiva;
\item[\textbf{(SMF3)}] 
$\forall f \in C^\infty(M),\,
\left\{
\begin{array}{c}
\{S,f\} = 0,\,
(H,f) = 0
\end{array}
\right. $
\end{description}} (cf. \cite{Mor86}).

Come già avvenuto per le strutture di Poisson in \cite{PeCa19} e \cite{CaPe23}, per le strutture di Nambu-Poisson in \cite{PeCa24a} e per le strutture di Dirac in \cite{PeCa24b}, si propone  qui di introdurre la nozione di struttura di Jacobi parziale su varietà convenienti o $c^\infty$-complete secondo Frölicher, Kriegl e Michor (cf. \cite{FrKr88} e \cite{KrMi97}). Tali strutture si collocano nel quadro generale delle strutture parziali compatibili nel contesto conveniente.\\
%\footnote{Il lettore troverà nell'articolo \cite{CaPe24} uno studio delle strutture bihamiltoniane parziali compatibili, ossia delle strutture \text{PQ}, \text{PN} e \text{P}$\Omega$.}.

Va inoltre sottolineato il legame stretto tra una struttura di Jacobi su una varietà $M$ e una struttura di Poisson omogenea sul fibrato lineare $M \times \R$, che permette di ottenere strutture di Jacobi su $M$ tramite la proiezione di una struttura di Poisson omogenea su $M \times \R$.\\

L'articolo è strutturato nel modo seguente.\\
Nella sezione~\ref{__VarietaDiJacobiDiDimensioneFinita} si richiama la nozione di varietà di Jacobi di dimensione finita, illustrata mediante esempi vari e le proprietà essenziali di tali strutture. La nozione di struttura parziale su varietà convenienti fa apparire un sottofibrato debole $T^\flat M$ del fibrato cotangente cinematico $T'M$ e un'algebra $\mathfrak{A}(M)$ di funzioni lisce su $M$. Nella sezione~\ref{__DerivazioniEParentesiDiSchoutenSuTflatM} viene definito la parentesi di Schouten per alcuni tensori antisimmetrici controvarianti, passo necessario alla definizione della nozione di struttura di Jacobi. Questo tipo di struttura è introdotto nella sezione~\ref{__VarietaDiJacobiParziali}, dove se ne studiano anche le proprietà, in particolare la distribuzione caratteristica ad essa associata. La sezione è inoltre arricchita da vari esempi.  
Nell'ultima sezione si propongono diverse direzioni di ricerca legate alla nozione di struttura di Jacobi parziale.

\section{Varietà di Jacobi di dimensione finita}
\label{__VarietaDiJacobiDiDimensioneFinita}

\subsection{Strutture di Jacobi}
\label{___StruttureDiJacobi}

La nozione di \emph{struttura di Jacobi}\index{struttura di Jacobi} su una varietà di dimensione finita $M$ definita in \cite{Kir76} è il dato di un'operazione bilineare
\[
\begin{array}{rccc}
\{.,.\} \colon	& C^\infty(M) \times C^\infty(M)	& \to	
							&	C^\infty(M)				\\
				& (f,g)								& \mapsto
							& \{f,g\}			
\end{array}
\] 
chiamata \emph{parentesi di Jacobi}\index{parentesi  di Jacobi} che soddisfa le seguenti proprietà: 
\begin{enumerate}
\item[(1)]
antisimmetria:\\
\[
\{g,f\} = - \{f,g\}
\]
\item[(2)]
identità di Jacobi:\\
\[
\{f,\{g,h\}\} + \{g,\{h,f\}\} + \{h,\{f,g\}\} = 0
\]
\item[(3)]
è \emph{locale}\index{parentesi locale}, i.e. il supporto\index{supporto} di $\{f,g\}$ è contenuto nell'intersezione dei supporti di $f$ e di $g$.
\end{enumerate}
La coppia $(M,\{.,.\})$ è detta \emph{algebra di Lie locale} \index{algebra di Lie locale}.\\

A. Kirillov ha mostrato che la parentesi di Jacobi può essere espressa come operatore differenziale di ordine al più uno in ciascun argomento.\\
Alloro esistono sulla varietà un campo vettoriale $E$, chiamato \emph{campo vettoriale di Reeb}\index{campo vettoriale di Reeb} e un tensore alternante $2$-contravariante $\Lambda$ definiti in maniera unica  dove, per ogni $f$ e ogni $g$ in $C^\infty(M)$, abbiamo:
\begin{eqnarray}
\label{eq_CrochetJacobi_LambdaR}
\{f,g\} = \Lambda(df,dg) + <f dg - g df,E>
\end{eqnarray}

A. Lichnerowicz ha introdotto il concetto mediante  l'esistenza di tali tensori $E$ e $\Lambda$ che soddisfano le condizioni di compatibilità
\[
\left\lbrace
\begin{array}{rcl}
[\Lambda,\Lambda]	&=&	2E \wedge \Lambda		\\
L_E \Lambda			&=&	0 
\end{array}
\right.
\] 
Ciò corrisponde al fatto che la parentesi (\ref{eq_CrochetJacobi_LambdaR}) soddisfa l'identità di Jacobi.\\

Siano $ \left( M_1,\Lambda_1,E_1 \right) $ e  
$ \left( M_2,\Lambda_2,E_2 \right) $ due varietà di Jacobi dotate rispettivamente delle parentesi
 $ \{.,.\}_{M_1}$ e $ \{.,.\}_{M_2}$.
Un'applicazione liscia $\varphi : M_1 \to M_2$ è una \emph{mappa di Jacobi}\index{mappa!di Jacobi} se l'applicazione indotta
\[
\begin{array}{cccc}
\varphi^\ast :& C^\infty(M_2) &\to &C^\infty(M_1)\\
& f & \mapsto & f \circ \varphi
\end{array}
\]
soddisfa la seguente proprietà:
\begin{equation}
\tag{\textbf{MJ}}
\forall (f,g) \in C^\infty(M_2)^2,\quad
\{ \varphi^\ast(f), \varphi^\ast(g) \}_{M_1}
= \varphi^\ast \left( \{f,g\}_{M_2} \right) .
\end{equation}

\subsection{Esempi}
\label{Ex_EsempiDiStruttureDiJacobi}

\begin{example}
\label{Varietà di Poisson}
{\sf Strutture di Poisson.}\\
Se $E=0$, si ritrova la nozione fondamentale di varietà di Poisson che generalizza quella di struttura simplettica.\\
Una \emph{struttura di Poisson} su una varietà di dimensione finita $M$ consiste nella definizione di una parentesi (applicazione bilineare) sull'algebra $C^\infty(M)$ antisimmetrica, che soddisfa l'identità di Jacobi e la regola di Leibniz:
\[
\{f,g.h\} = \{f,g\}h + g\{f,h\}.
\]
Nel caso di tale struttura, la dinamica è descritta dall'evoluzione temporale di un osservabile $x$ in funzione del tempo. Se $H$ è l'Hamiltoniana del sistema, si ha:
\[
\dot{x}(t) = \{x(t),H(t)\}.
\]
Notiamo che, poiché la parentesi di Poisson è antisimmetrica, si ha in particolare:
\[
\frac{dH}{dt} = \{H,H\} = 0
\]
e quindi l'energia $H$ è conservata.\\

Un esempio fondamentale di tali strutture è quello delle strutture di Lie-Poisson, che rappresentano le strutture di Poisson lineari. Esse svolgono inoltre un ruolo fondamentale nella meccanica Hamiltoniana, ad esempio nelle equazioni di Eulero per il corpo rigido $\mathfrak{so}(3)^\ast$ (cf. \cite{MaRa99}).\\
Sia $M=\mathfrak{g}^\ast$ il duale di un'algebra di Lie $\mathfrak{g}$ di dimensione finita $n$.\\
Per due funzioni $f$ e $g$ di $C^\infty(M)$ si definisce la \emph{parentesi di Lie-Poisson} come
\[
\{f,g\}(\alpha)
=\langle \alpha,[,df_\alpha,dg_\alpha,]\rangle,
\]
dove $\alpha \in \mathfrak{g}^\ast$ e dove si identificano le differenziali $df_\alpha$ e $dg_\alpha$ con elementi di $\mathfrak{g}$.\\
In questo modo si ottiene una struttura di Poisson $P$, chiamata \emph{struttura di Lie-Poisson} oppure \emph{struttura KKS} (Kirillov-Kostant-Souriau). Nella base associata alle coordinate globali $ \left( a^k \right) $ su $\mathfrak{g}^\ast$, essa si esprime come:
\[
P=\displaystyle\sum_{1\leqslant i<j \leqslant n}\sum_{k=1}^n 
c_{ij}^k \alpha_k
\frac{\partial f}{\partial x_i}\frac{\partial g}{\partial x_j}
\]
dove $c_{ij}^k$ sono le costanti di struttura dell'algebra di Lie $\mathfrak{g}$.
\end{example}

\begin{example}
\label{Ex_StructureDeJacobiSurR2m+1}
{\sf Struttura standard di Jacobi su $\R^{2m+1}$.}\\
Consideriamo lo spazio vettoriale $\R^{2m+1}$ munito del sistema di coordinate canoniche 
$ \left( x^0,x^1,\dots ,x^{2m} \right) $.\
Definiamo dunque la struttura standard di Jacobi su questo spazio mediante:
\begin{enumerate}
\item[$\bullet$]
il campo di Reeb:
$E = \dfrac{\partial}{\partial x^0}$
\item[$\bullet$]
il tensore di Jacobi:
$\Lambda
=\displaystyle\sum_{i=1}^m
\left( x^{m+i}\frac{\partial}{\partial x^0}
-\frac{\partial}{\partial x^i}
\right)
\wedge
\frac{\partial}{\partial x^{m+i}}$.
\end{enumerate}
La parentesi di Jacobi associata è quindi definita, per ogni coppia di funzioni lisce $(f,g)$, da
\[
\{f,g\} = \Lambda(df,dg) + f E(g) - g E(f)
\]
la cui espressione in coordinate locali è (cf. \cite{Lic78}):
\[
\{f,g\}
=\displaystyle\sum_{i=1}^m
\left(
\left( x^{m+i}\dfrac{\partial f}{\partial x^0}
-\dfrac{\partial f}{\partial x^i}
\right)
\dfrac{\partial g}{\partial x^{m+i}}
-
\left( x^{m+i}\dfrac{\partial g}{\partial x^0}
-\dfrac{\partial g}{\partial x^i}
\right)
\dfrac{\partial f}{\partial x^{m+i}}
\right) 
+f\dfrac{\partial g}{\partial x^0}
- g\dfrac{\partial f}{\partial x^0}.
\]
Il lettore potrà trovare in \cite{Cab10} una stratificazione delle varietà di Jacobi generiche di dimensione dispari.
\end{example}

\begin{example}
\label{Ex_VarietaCosimplettica}
{\sf Varietà cosimplettica}\\
Una \emph{varietà cosimplettica}\index{varietà!cosimplettica} è una terna $(M,\Omega,\eta)$ dove $M$ è una varietà di dimensione dispari $2m+1$, $\Omega$ è una $2$-forma chiusa e $\eta$ è una $1$-forma chiusa su $M$, tali che
$\eta \wedge \Omega^{m}$ 
sia una forma di volume.\\
Se $\flat : \chi(M) \rightarrow \chi^{\ast}(M)$
è l'isomorfismo di $C^{\infty}(M)$-moduli definito da
\[
\flat(V) = i_{V}\Omega + (i_{V}\eta) \eta ,
\]
il campo vettoriale $E = \flat^{-1}(\eta)$ è il campo di Reeb su $M$; esso è caratterizzato dalle seguenti relazioni: 
\[
i_{E}\Omega = 0 \text{~~e~~} i_{E}\eta = 1.
\] 
In particolare, si ha: 
\[
L_{E}\Omega = 0 \text{~~e~~} L_{E}\eta = 0.
\]
Il $2$-tensore $P$ definito da
\[
P(\alpha,\beta) 
= \Omega\bigl(\flat^{-1}(\alpha),\flat^{-1}(\beta)\bigr)
\]
dota la varietà $M$ di una struttura di Poisson per la quale vale
\[
L_E P= 0.
\]
L'esempio standard di varietà cosimplettica è fornito dal fibrato cotangente esteso 
$ \left( T^\ast N \times \R, \text{d}t, \pi^\ast \Omega \right) $ 
dove $\pi:T^\ast N \times \R \to T^\ast N$ è la   proiezione canonica, e $\Omega$ è la forma simplettica canonica su $T^\ast N$ (cf. \cite{Alb89}).\\
Il quadro simplettico (risp. cosimplettico) modella sistemi hamiltoniani autonomi (risp. dipendenti dal tempo). In entrambi i casi si tratta di sistemi conservativi.
\end{example}

\begin{example}
\label{Ex_VarietaDiContatto}
{\sf Varietà di contatto}\\
Le varietà di contatto costituiscono esempi classici di varietà di Jacobi. Esse trovano applicazione anche come quadro geometrico nella meccanica (cf. \cite{DeLVa19}) e nella termodinamica (cf. \cite{Mru95}) e forniscono un quadro per sistemi non conservativi.\\

Siano $M$ una varietà di dimensione $2m+1$ e $\theta$ una $1$-forma su $M$.
Si dice che $\theta$ è una \emph{forma di contatto}\index{forma!di contatto} se 
$ 
\theta \wedge (\text{d}\theta)^m
$ 
è non nulla in ogni punto. Una \emph{varietà di contatto} è dunque una varietà munita di una forma di contatto.\\
In un intorno di ogni punto esiste un sistema di coordinate (canoniche)\\
$
\left( t,q^1,\dots,q^m,p_1,\dots,p_m \right)
$ 
tale che la forma di contatto si scriva come:
\[
\theta = \text{d}t 
- \displaystyle\sum_{i=1}^{m} p_i  \text{d}q^i.
\]
Una varietà di contatto può essere dotata di una struttura di varietà di Jacobi, in cui il $2$-tensore $\Lambda$ è definito, per tutte le $1$-forme $\alpha$ e $\beta$, da:
\[
\Lambda(\alpha,\beta) 
= \text{d}\alpha 
\big(\flat^{-1}(\alpha), \flat^{-1}(\beta)\big),
\]
dove
$\flat : \mathfrak{X}(M) \to \mathfrak{X}^\ast(M)$ è l'isomorfismo di moduli $C^\infty(M)$ definito da:
\[
\flat(V) = i_V \text{d}\theta + (i_V \theta) \theta.
\]
Il campo di Reeb $E$ associato è caratterizzato dalle relazioni
\[
i_E \theta = 1~~~\text{e}~~~ i_E d\theta = 0.
\]
Nelle coordinate canoniche sopra definite si ottiene allora l'espressione di $\Lambda$ e di $E$:
\[
\left\lbrace
\begin{array}{l}
\Lambda = \displaystyle\sum_{i=1}^{m} \Big( \dfrac{\partial}{\partial q^i} + p_i \dfrac{\partial}{\partial t} \Big) \wedge \dfrac{\partial}{\partial p_i}\\
E = \dfrac{\partial}{\partial t}.
\end{array}
\right.
\]
\end{example}

\begin{example}
\label{Ex_VarietaDiJacobiSulFibratoDei1Jet}
{\sf Varietà di Jacobi sul fibrato dei $1$-jet.}

Sia $E$ un fibrato in linee sulla varietà $M$ di dimensione $n$. Denotiamo con $J^{1}(E)$ il fibrato dei jet di ordine $1$ di $E$ (cf. \cite{Sau89}). Se $  \left( x^i \right) _{1\leqslant i \leqslant n} $ sono coordinate locali su un aperto $U$ della base $M$ e $u$ è una coordinata sulla fibra, allora 
$J^1(E)$ è dotato delle coordinate 
$ \left( x^1,\dots,x^n,u,u_1,\dots,u_n \right) $ 
dove $ u_i=\dfrac{\partial u}{\partial x^i}$.
\\
La \emph{distribuzione di Cartan}\index{distribuzione!di Cartan} è localmente generata sopra $U$ dai campi vettoriali locali
\[
X_{i}=\dfrac{\partial}{\partial x^{i}}
+u_{i}\dfrac{\partial}{\partial u}.
\]
Essa costituisce il nucleo della forma di contatto $\theta$, la cui espressione in queste coordinate locali è
\[
\text{d}u-\displaystyle\sum_{i=1}^{n}u_i 
\text{d}x^i.
\]
L'espressione locale della parentesi di Jacobi, per due funzioni reali $f$ e $g$ di $J^1(E)$, è allora:
\[
\{f,g\}
=
\displaystyle\sum_{i=1}^n
\left(
\dfrac{\partial f}{\partial x^i}
\dfrac{\partial g}{\partial u_i}
- \dfrac{\partial g}{\partial x^i}
\dfrac{\partial f}{\partial u_i}
\right) 
+ f \dfrac{\partial g}{\partial u}
- g \dfrac{\partial f}{\partial u}.
\]
Il lettore troverà un esempio di struttura di Jacobi sullo spazio $J^{\infty}(M)$ dei jet di ordine infinito delle funzioni $f : M \to \R$, dove $M$ è una varietà di dimensione finita, in \cite{LiZh11}.
\end{example}

\subsection{Algebroidi di Jacobi}
\label{___AlgebroidiDeJacobi}

È ben noto che, in dimensione finita, esiste una corrispondenza biunivoca tra le strutture di algebroide di Lie su un fibrato vettoriale $A$ e le strutture di Poisson lineari sul duale $A^\ast$ (cf. \cite{CDW87})\footnote{Una classe importante di algebroidi di Lie è quella dei bialgebroidi di Lie $A$, in cui  $A$ e $A^\ast$ sono dotati di strutture di algebroide di Lie compatibili in un certo senso (cf. \cite{Kos95}). Se $(M,P)$ è una varietà di Poisson di dimensione finita, allora la coppia $ \left( TM,T^\ast M \right) $ è un bialgebroide di Lie. In senso inverso, è stato dimostrato in \cite{MaXu94} che la base di un bialgebroide è una varietà di Poisson.}.\\
Se $M$ è una varietà di Jacobi di dimensione finita, allora il fibrato $T^\ast M$ in generale non può essere dotato di una struttura di algebroide di Lie. Tuttavia, il fibrato dei $1$-jet $T^\ast M \times \R \to M$ ammette una struttura di algebroide di Lie (cf. \cite{KeSo93}).\\
La nozione di algebroide di Jacobi è stata introdotta da J. Grabowski e G. Marmo nel loro articolo \cite{GrMa01}, in cui evidenziano il legame tra strutture di Jacobi e algebroidi di Lie dotati di un cociclo.

\begin{definition}
\label{D_AlgebroideDeJacobi}
Un \emph{algebroide di Jacobi}\index{algebroide di Jacobi} è un algebroide di Lie $(A,[.,.],\rho)$ munito di un $1$-cociclo $\phi \in \Gamma \left( A^\ast \right)$ tale che
\[
[X,fY]=f[X,Y]+(\rho(X)f)Y-\phi(X)fY.
\]
\end{definition}

In \cite{Vit18}, L. Vitagliano definisce questa nozione tramite un fibrato in linee.\\

Una struttura di Jacobi $(\Lambda,E)$ su una varietà $M$ induce naturalmente un algebroide di Jacobi su $T^\ast M \oplus \R$ (cf. \cite{IgMa01}).

\section{Derivazioni e parentesi di Schouten su $T^\flat M$}
\label{__DerivazioniEParentesiDiSchoutenSuTflatM}

Lo scopo di questa sezione è introdurre la parentesi di Schouten su alcune sezioni di un sottofibrato del fibrato cotangente di una varietà conveniente al fine di definire la nozione di varietà di Jacobi parziale.\\

Si fanno ampio uso dei risultati ottenuti nel libro \cite{CaPe23}.\\

Siano $M$ una varietà modellata sul spazio  vettoriale conveniente $\mathbb{M}$ (cf. \cite{KrMi97}, 27.1), $p_{TM}:TM\to M$ su fibrato tangente cinematico (cf. \cite{KrMi97}, 28.12) e $p_{T'M}:T'M\to M$ su fibrato cotangente cinematico (cf. \cite{KrMi97}, 33.1). 

\subsection{L'algebra $\mathfrak{A} (U)$}
\label{__LAlgebramathfrakAU}

\begin{definition}
\label{D_SottoFibratoDebole}
Un sottofibrato $p^{\flat}:T^{\flat}M\to M$ di $p_{M}^{\prime}:T^{\prime}M \to M$ dove $p^{\flat}:T^{\flat}M\to M$ e un fibrato conveniente, e un \emph{sottofibrato debole} \index{debole!sotto-fibrato} di $p_{M}^{\prime}:T^{\prime}M\to M$ se l'iniezione  canonica $\iota:T^{\flat}M \to T^{\prime}M$ è un morfismo de fibrati convenienti.
\end{definition}

Facendo riferimento a \cite {KrMi97}, Definition 48.5, si introduce il seguente insieme.
\begin{definition}
\label{D_AU}
Per ogni aperto $U$ di $M$, si considera l'insieme $\mathfrak{A}(U)$ \index{AmathfrakU@$\mathfrak{A}(U)$} delle funzioni $f \in C^\infty(U)$ tale che, per ogni intero naturale non nullo $k$ e ogni $x$ di $U$, la derivata di ordine $k$ di $f$ in $x$, $d^{k}f(x)\in L_{\operatorname{sym}}^{k}(T_{x}M,\mathbb{R})$ soddisfa:
\begin{equation}
\label{eq_dkf}
\forall (u_2,\dots,u_k) \in (T_xM)^{k-1},\;
d^{k}_xf(.,u_{2},\dots,u_{k}) \in T_{x}^{\flat}M.
\end{equation}
\end{definition}

\begin{proposition}
\label{P_AUalgebra}
Sia $U$ un aperto di $M$.
\begin{enumerate}
\item
L'insieme $\mathfrak{A} (U)$ è una sottoalgebra di $C^\infty(U)$.
\item
Per qualsiasi intero naturale $k$ e tutti i campi vettoriali locali $X_1, \dots ,X_k$ sopra $U$, 
l'applicazione $x \mapsto d^k f(X_1, \dots,X_k)(x)$ appartiene a $\mathfrak{A} (U)$.
\end{enumerate}
\end{proposition}
\begin{proof}
cf. \cite{CaPe23}, 7.1.1.
\end{proof}

\subsection{Parentesi di Schouten su $T^\flat M$}
\label{__ParentesiDiSchoutenSuTflatM}

Useremo la definizione della parentesi di Schouten su una varietà di Poisson come data in  \cite{FeMa14}, 1.4, per proporre una generalizzazione della parentesi di Schouten su $T^\flat M$.\\

\begin{definition}
\label{D_Kinematic kAlternatingDerivation} Sia $U$  un aperto di $M$.
\begin{enumerate}
\item[1.]
Se $k\geq 1$, una derivazione $k$-alternante \index{derivazione $k$-alternante} di $\mathfrak{A}(U)$ è un'applicazione $k$-lineare alternante limitata $D:(\mathfrak{A}(U))^k\to \mathfrak{A}(U)$ per cui
%\item
\[
\begin{array}{rl}
&D(f_1,\dots,f_{i-1},gh,f_{i+1}\dots, f_k)\\
=& gD(f_1,\dots,f_{i-1},h,f_{i+1}\dots, f_k)+D(f_1,\dots,f_{i-1},g,f_{i+1}\dots, f_k)h
\end{array}
\]
per ogni $i \in \{1,\dots,k\}$ ed ogni $f_1,\dots, f_1,\dots,f_{i-1},g,h,f_{i+1}\dots, f_k$ in $\mathfrak{A}(U)$.
\item[2.]
Una derivazione $k$-alternante $D$ di $\mathfrak{A}(U)$ sarà chiamata \emph{di ordine $1$} se $D(f_1,\dots, f_k)$ dipende soltanto dal $1$-jet di ciascuna $f_i$ per $i \in \{1,\dots,k\}$.
\end{enumerate}
\end{definition}

\begin{description}
\item[$\bullet$]
Lo spazio delle derivazioni $k$-alternanti sarà indicato con $\operatorname{Der}_k \left( \mathfrak{A}(U) \right) $\index{DerkAU@$\operatorname{Der}_k \left( \mathfrak{A}(U) \right) $ (spazio delle derivazioni $k$-alternanti di $\mathfrak{A}(U)$)}.
\item[$\bullet$]
Lo sottospazio delle derivazioni $k$-alternanti di ordine $1$ sarà indicato con $\operatorname{Der}^1_k(\mathfrak{A}(U))$\index{DerkAU1@$\operatorname{Der}_k^1(\mathfrak{A}(U))$ (spazio delle  derivazioni $k$-alternanti di $\mathfrak{A}(U)$ di ordine $1$)} .
\end{description}

\begin{definition}
\label{D_DerivazioneKinematica}
Una derivazione $D\in \operatorname{Der}_{k }(\mathfrak{A}(U))$ di ordine $1$ è chiamata una derivazione $k$-alternante cinematica\index{kAtlternanteDerivazione@derivazione $k$-alternante cinematica} di $\mathfrak{A}(U)$ se, per ogni $f_2,\dots,f_{k}$ fissati in $\mathfrak{A}(U)$, esiste un campo vettoriale $X$ su $U$ tale che:
\begin{equation}
\label{eq_DerivazioneCinematica}
D(f,f_2,\dots,f_k)=df(X)
\end{equation}
per ogni $f\in \mathfrak{A}(U)$.
\end{definition}

\begin{description}
\item[$\bullet$]
Denoteremo con $\mathbf{D}_k(\mathfrak{A}(U))$\index{DkmathfrakAU@$\mathbf{D}_k(\mathfrak{A}(U))$ (insieme delle derivazioni $k$-alternanti cinematiche di $\mathfrak{A}(U)$)} l'insieme delle derivazioni $k$-alternanti cinematiche di $\mathfrak{A}(U)$.
\end{description}

\begin{remark}
In dimensione finita, tutte le derivazioni di $C^\infty(U)$  sono cinematiche (cf. \cite{FeMa14}). Ciò non è più vero per gli spazi di Banach.
\end{remark}

Come in \cite{FeMa14}, 1.4,, introduciamo
\begin{definition}
\label{D_DcircD'}
Se $D \in \operatorname{Der}_{k}(\mathfrak{A}(U))$ e $D' \in \operatorname{Der}_{k' }(\mathfrak{A}(U))$
\begin{enumerate}
\item
si ha
\[
\begin{array}{cl}
	&D\circ D'(f_1,\dots, f_{k'}, f_{k'+1}, \dots ,f_{k+k'-1})\\
	=&\displaystyle\sum_{\sigma}(-1)^{\operatorname{sign}(\sigma)} D\left( D'(f_{\sigma(1)},\dots, f_{\sigma(k')}),f_{\sigma(k'+1)},\dots, f_{\sigma_{(k+k'-1)}}) \right)
\end{array}
\]
per ogni $f_i \in \mathfrak{A}(U)$, $i \in \{ 1,\dots k+k'-1 \} $ dove $\sigma$ corrisponde a tutte le $(k+k'-1)$-uple tale che $\sigma(1)<\cdots<\sigma(k')$ e $\sigma(k'+1)<\cdots<\sigma(k+k'-1)$.
\item
\begin{eqnarray}
\label{eq_ParentesiDerivation}
\index{parentesi!di derivazioni}
[D,D']= D\circ D' -(-1)^{(k-1)(k'-1)} D'\circ D.
\end{eqnarray}
\item[(3)]
Il prodotto esterno  \index{prodotto esterno}  $D\wedge D'$ di $D$ e di $D'$ è definito da:
\begin{eqnarray}
\label{eq_WedgeProduct_DerivativeOperators}
& & D\wedge D'( f_1,\dots, f_{k+k'})\\
& =&\frac{1}{k!}\frac{1}{k'!}\sum_{\sigma} (-1)^{{\rm sign}\sigma}D(f_{\sigma(1)},\dots, f_{\sigma(k)})D'(f_{\sigma(k+1)},\dots,f_{\sigma(k+k')})
\nonumber
\end{eqnarray}
dove $\sigma$ corrisponde a tutte le $(k+k')$-uple tale che $\sigma(1)<\cdots<\sigma(k)$ e $\sigma(k+1)<\cdots<\sigma(k+k')$.
\end{enumerate}
\end{definition}

\begin{proposition}
\label{P_PArentesiDerivation}
Sia $U$ un aperto di $M$. Si ha le seguenti proprietà:
\begin{enumerate}
\item
${Der}_k(\mathfrak{A}(U))$  a una struttura di  $\mathfrak{A}(U)$-modulo e  ${Der}^1_k(\mathfrak{A}(U))$ è un sottomodulo.
\item
Se $D$ appartiene  a $\operatorname{Der}_{k}(\mathfrak{A}(U))$ e $D'$ a $\operatorname{Der}_{k' }(\mathfrak{A}(U))$ allora
$ [D, D']$ appartiene a  $\operatorname{Der}_{(k+k'-1)}(\mathfrak{A}(U))$.
\item
La parentesi $[.,.]$ è $\mathbb{R}$ bilineare su  $\operatorname{Der}(\mathfrak{A}(U))$ e si ha le seguenti proprietà:
\begin{enumerate}
\item[(i)] \hfil{}
$
[D,D']=-(-1)^{(k-1)(k'-1)}[D,D'].
$
\item[(ii)]
(Identità di Jacobi generalizzata)\index{Identità di Jacobi generalizzata}\\
Per ogni $D\in \operatorname{Der}_{k}(\mathfrak{A}(U))$, $D'\in \operatorname{Der}_{k'}(\mathfrak{A}(U))$ e $D''\in \operatorname{Der}_{k''}(\mathfrak{A}(U))$,
\[
\begin{array}{l}
(-1)^{(k-1)(k''-1)}[[D,D'],D'']
+(-1)^{(k'-1)(k-1)}[[D',D''],D]\\
+(-1)^{(k''-1)(k'-1)}[[D'',D],D']=0.
\end{array}
\]
\end{enumerate}
\end{enumerate}
\end{proposition}

Il fibrato vettoriale
\[
p_k^\flat : 
L^{k}_{\operatorname{alt}} \left( T^\flat M,\mathbb{R} \right) 
=
\displaystyle\bigcup_{x \in M}
L^{k}_{\operatorname{alt}} \left( T_x^\flat M,\mathbb{R} \right) 
\]
dove $L^{k}_{\operatorname{alt}} \left( T_x^\flat M,\mathbb{R} \right) $ è il spazio vettoriale  de tutte le applicazioni $k$-lineari alternanti limitate $T_x^\flat M \to \mathbb{R}$, è conveniente.\\
Il fibrato vettoriale
\[
q_k^\flat : 
L^{k}_{\operatorname{alt}} \left( T^\flat M,TM \right) 
=
\displaystyle\bigcup_{x \in M}
L^{k}_{\operatorname{alt}} \left( T_x^\flat M,T_x M \right) 
\]
dove $L^{k}_{\operatorname{alt}} \left( T_x^\flat M,T_x M \right) $ è il spazio vettoriale  de tutte le applicazioni $k$-lineari alternanti limitate $T_x^\flat M \to T_x M$, è conveniente.
\begin{enumerate}
\item[$\bullet$]
Il spazio vettoriale delle sezioni locali di $p_k^\flat$ sopra il aperto $U$ se denota con 
$\bigwedge^{k}\Gamma^* \left( T^\flat M_U \right) $.
\item[$\bullet$]
Il spazio vettoriale delle sezioni locali di $q_k^\flat$ sopra il aperto $U$ se denota con 
$\bigwedge^{k}\Gamma^* \left( T^\flat M_U,TM_U \right) $.
\end{enumerate}
L'insieme
\[
\left\{
\bigwedge^{k}\Gamma^* \left( T^\flat M_U,\mathbb{R} \right) ,\; U \textrm{ aperto in } M
\right\}
\]
è un fascio di moduli sul fascio $C^\infty(.)$.\\
Per $k \geq 1$, una sezione $P \in \bigwedge^{k}\Gamma^* \left( T^\flat M_U,\mathbb{R} \right) $ è caratterizzata  dalle valori $P \left( df_1,\dots,df_k \right) $ dove $ \left( f_1,\dots,f_k \right) \in  \left( \mathfrak{A}(U) \right) ^k$.\\

Se $\iota : T^\flat M \to T'M$ è il morfismo di inclusione, allora $\iota^\ast : T''M \to \left( T^\flat M \right)' $ è un morfismo de fibrati.\\
$TM$ è un sottofibrato di $T''M$.

\begin{definition}
\label{D_AmmissibileP}
Sia $U$ un aperto di $M$.
\begin{enumerate}
\item[(i)]
Per $k=1$, un elemento  $\Lambda\in \bigwedge^1\Gamma^*(T^\flat M_U)=\Gamma^*(T^\flat M_U)$ è \emph{ammissibile}\index{derivazione ammissibile}\\
se esiste un campo vettoriale $X$ su $U$ tale che $\Lambda=\iota^* X$.
\item[(ii)]
Per $k\geq 2$, una sezione $\Lambda \in \bigwedge^{k}\Gamma^*(T^\flat M_U )$ è \emph{ammissibile}\\
se esiste  $\Lambda^\sharp\in \bigwedge^{k-1}\Gamma^*(T^\flat M_U,TM_U)$ tale che
\[
\Lambda_x(\alpha_1, \alpha_2,\dots,\alpha_k)=<\alpha_1, \Lambda^\sharp_x( \alpha_2,\dots,\alpha_k)>
\]
per ogni $\alpha_1,\dots, \alpha_k \in T_x^\flat M$.
\end{enumerate}
\end{definition}

\begin{description}
\item[$\bullet$]
Il spazio degli elementi ammissibili di $\bigwedge^{k}\Gamma^* \left( T^\flat M_U \right) $ sarà indicato con $\Gamma^*_k \left( \mathfrak{A}(U) \right) $\index{GammaStarkAU@$\Gamma^*_k \left( \mathfrak{A}(U) \right) $ (spazio degli elementi ammissibili di $\bigwedge^{k}\Gamma^* \left( T^\flat M_U \right) $)}. 
\item[$\bullet$]
Il  spazio delle derivazioni ammissibili di ${\bf D}_k \left( \mathfrak{A}(U) \right) $ sarà indicato con ${\bf D}^*_k \left( \mathfrak{A}(U) \right) $\index{DStarkboldAU@${\bf D}^*_k \left( \mathfrak{A}(U) \right) $ (spazio delle derivazioni  ammissibili di ${\bf D}_k \left( \mathfrak{A}(U) \right) $)}.
\end{description}

\begin{proposition}
\label{P_AsociazioneLambdaD}
Ad ogni $\Lambda\in \Gamma^*_k(\mathfrak{A}(U))$ è associata una derivazione $k$-alternante cinematica $D_\Omega\in {\bf D}_k(\mathfrak{A}(U))$
definita da
\begin{equation}
\label{eq_DLambda}
D_\Lambda (f_1, \dots, f_k)=\Lambda(df_1, \dots, df_k).
\end{equation}
per ogni $ \left( f_1,\dots,f_k \right) \in \mathfrak{A}(U)^k$.\\
\end{proposition}
L'applicazione
$
\begin{array}[c]{ccc}
\Gamma^*_k(\mathfrak{A}(U))	& \to 	
	& {\bf D}_k(\mathfrak{A}(U))		\\
\Lambda    & \mapsto 	& D_\Lambda
\end{array}
$ 
è iniettiva, ma non suriettiva in generale.

\begin{proposition}
\label{P_Parentesi}
Sia $U$ un aperto di $M$. Allora per $\Lambda \in \Gamma^*_k(\mathfrak{A}_U)$ e $\Omega\in \Gamma^*_l(\mathfrak{A}_U )$,
la parentesi $ [D_\Lambda,D_\Omega]$ è una derivazione $(k+l-1)$-alternante cinematicha di $\mathfrak{A}(U)$ e esiste un unico elemento  $[\Lambda,\Omega]_S\in \Gamma^*_{k+l-1}(\mathfrak{A}_U)$ tale che
\[
D_{[\Lambda,\Omega]_S}=[D_\Lambda,D_\Omega].
\]
\end{proposition}
El elemento $[\Lambda,\Omega]_S$ di $\Gamma^*_{k+l-1}(\mathfrak{A}_U)$ è chimato \emph{parentesi di Shouten}\index{parentesi!di Scouthen} di $\Lambda$ e $\Omega$.

\begin{theorem}
\label{T_ProprietaDellaParentensiDiSchouten}
La parentesi di Schouten ha le seguenti proprietà:
\begin{enumerate}
\item
Per ogni campi vettoriali $X$ e $Y$ su $U$, $ \iota^*X$ and $\iota^*Y$ appartiene  a$\Gamma^*_1(\mathfrak{A}(U))$ e si ha
\[
\iota^*[X,Y]=[\iota^*X,\iota^*Y]_S.
\]
\item
Per ogni $\Omega \in \Gamma^*_k \left( \mathfrak{A}(U) \right) $ e $\Phi\in \Gamma^*_h \left( \mathfrak{A}(U) \right) $,
\[
[\Omega, \Phi]_S=-(-1)^{(k-1)(h-1)}[\Phi, \Omega]_S.
\]
\item
Per ogni  $\Omega \in \Gamma^*_k \left( \mathfrak{A}(U) \right) $,  $\Phi\in \Gamma^*_h \left( \mathfrak{A}(U) \right) $ e $\Psi \in \Gamma^*_l \left( \mathfrak{A}(U) \right) $,
\[
[\Omega,\Phi\wedge \Psi]_S=[\Omega, \Phi]_S\wedge \Psi+(-1)^{(k-1)h}\Omega\wedge [\Phi, \Psi]_S.
\]
\item
Per ogni  $\Omega \in \Gamma^*_k \left( \mathfrak{A}(U) \right) $,  $\Phi\in \Gamma^*_h \left( \mathfrak{A}(U) \right) $ e $\Psi\in  \Gamma^*_l \left( \mathfrak{A}(U) \right) $,
\[
(-1)^{(k-1)(l-1)}[\Omega,[\Phi, \Psi]_S]_S+(-1)^{(h-1)(k-1)}[\Phi,[ \Psi, \Omega]_S]_S+(-1)^{(l-1)(h-1)}[\Psi,[ \Omega, \Phi]_S]_S=0.
\]
[Identità di Jacobi generalizzata]\index{identità!di Jacobi generalizzata}.
\item
Se $X_1\wedge\cdots\wedge X_k$ e  $Y_1\wedge\cdots\wedge Y_h$ sono multivettori, allora $\iota^*(X_1\wedge\cdots\wedge X_k)$ e $\iota^*(Y_1\wedge\cdots\wedge Y_h)$ appartengono a $\Gamma_k^*(\mathfrak{A}(U))$  e $\Gamma_h^*(\mathfrak{A}(U))$ rispettivamente e si ha
\[
\begin{array}{cl}
&[\iota^*(X_1\wedge\cdots\wedge X_k),\iota^*(Y_1\wedge\cdots\wedge Y_h)]_S\\
=&
\iota^*
\left( \displaystyle \sum_{i,j}(-1)^{i+j}[X_i,Y_j]X_1\wedge\cdots\wedge\widehat{X_i}\wedge\cdots\wedge X_k\wedge Y_1\wedge\cdots\wedge \widehat{Y_j}\wedge\cdots\wedge Y_h \right)
\end{array}
\]
\end{enumerate}
\end{theorem}
\begin{proof}
cf. \cite{CaPe23}, 7.1.1.
\end{proof}

\section{Varietà di Jacobi parziali}
\label{__VarietaDiJacobiParziali}

Il concetto di varietà di Jacobi parziale introdotta qui costituisce una generalizzazione della nozione  di varietà di Jacobi in dimensione finita e di varietà di Poisson parziale conveniente, definita da F. Pelletier in \cite{PeCa19}.\\

In questa sezione, $M$ è una varietà modellata sul spazio  vettoriale conveniente $\mathbb{M}$, $p_{TM}:TM\to M$ è su fibrato tangente cinematico e $p_{T'M}:T'M\to M$ su fibrato cotangente cinematico.\\
Sia $p^{\flat}:T^{\flat}M\to M$ un sottofibrato debole di $p_{TM'}:T'M \to M$. Consideriamo l'iniezione  canonica $\iota:T^{\flat}M \to T'M$ e $\iota^\ast : T''M \to \left( T^\flat M \right)' $ che sono morfismi de fibrati convenienti.\\ 
Sia, per ogni aperto $U$ di $M$, $\mathfrak{A} (U)$ la sottoalgebra di $C^\infty(U)$ definita in $\S$~\ref{__ParentesiDiSchoutenSuTflatM}.\\
Consideriamo anche la parentesi di Schouten su $\displaystyle\bigcup_{k \in \mathbb{N}^\ast}\Gamma^*_k(\mathfrak{A}_U)$.
 
\subsection{Definizioni}

\begin{definition}
\label{D_VarietaDiJacobiParziale}
La terna $(M,\Lambda,X)$ dove $\Lambda \in \Gamma^*_k \left( \mathfrak{A}(U) \right) $ e $X$ è un campo vettoriale è una \emph{varietà di Jacobi parziale}\index{varietà!di Jacobi parziale} se $\Lambda$ e $X$ soddisfano le seguenti proprietà:
\begin{description}
\item[\textbf{(VJp1)}]
{$\hfil 
[\Lambda,\Lambda]_S = 2 \iota^\ast(X) \wedge \Lambda$
}
\item[\textbf{(VJp2)}]
{$\hfil
L_X \Lambda = 0$
}
\end{description}
\end{definition}

\begin{remark}
\label{R_LocalisazioneVarietaDiJJacobiarziali}
Un problema che va sottolineato è che, a differenza del contesto a dimensione finita, una funzione liscia locale su una varietà conveniente $M$ non necessariamente si estende a una funzione globale su $M$ se non esistono alcuni tipi di funzioni a supporto compatto (\textit{bump functions}).\\ 
Pertanto, l'algebra $C^\infty(M)$ delle funzioni lisce su $M$, ristretta a un aperto $U$, può essere strettamente contenuta nell'algebra $C^\infty(U)$ delle funzioni lisce su $U$.\\
Una situazione analoga si presenta anche nel contesto delle varietà di Banach.\\
Poiché molti esempi classici di varietà convenienti non ammettono tali funzioni, la nozione di struttura di Jacobi parziale avrebbe potuto essere definita su insiemi di funzioni lisce definite su aperti di $M$ (cf. \cite{PeCa24a}, dove viene sollevato questo problema nel caso delle strutture di Nambu-Poisson).
%La ragione essenziale di tali considerazioni locali è che la maggior parte delle proprietà delle strutture di Jacobi finite-dimensionali dipende dal germe di tali strutture di algebre di Jacobi.
\end{remark}

Se $(M,\Lambda,X)$ è una varietà di Jacobi, la \emph{parentesi di Jacobi}\index{parentesi!di Jacobi} $\{.,.\}$ è definita per ogni coppia $ \left( f,g \right) $ di funzioni di $\mathfrak{A}(M)$ da
\[
\{f,g\} = \Lambda(df,dg) + f X(g) - g X(f)
\] 
Questa applicazione è bilineare, antisimmetrica e soddisfa l'identità di Jacobi.\\

\smallskip
In considerazione dell'osservazione~\ref{R_LocalisazioneVarietaDiJJacobiarziali}, ci poniamo nella situazione in cui la parentesi è localizzabile, cioè per tutti gli aperti $U$ e $V$ di $M$, si ha:

$(\{.,\dots,.\}_U)_{| U\cap V}=(\{.,\dots,.\}_V)_{| U\cap V}=\{.,\dots,.\}_{ U\cap V}$

$(\{.,\dots,.\}_{ U\cap V})_{| U}=\{.,\dots,.\}_U\;\;\; (\{.,\dots,.\}_{ U\cap V})_{| V}=\{.,\dots,.\}_V$.\\
\smallskip

Poiché $\Lambda \in
\Gamma^*_2(\mathfrak{A}(M)) $, esiste 
$\Lambda^\sharp \in \bigwedge^1 \Gamma^*(T^\flat M_U,TM_U)$ tale che, per ogni $x \in M$
\[
\Lambda_x(\alpha,\beta)
=<\alpha, \Lambda^\sharp_x(\beta)>
\]
La parentesi risulta dunque:
\[
\{f,g\} = <\alpha, \Lambda^\sharp_x(\beta)> 
+ f X(g) - g X(f).
\]
In base alla Proposizione~\ref{P_AUalgebra}, 2., abbiamo $\{f,g\} \in \mathfrak{A}(U)$.\\  

A ogni funzione $f$ di $\mathfrak{A}(M)$, si può associare il campo di vettori Hamiltoniano
\[
X_f = \Lambda^\sharp(df) + fX
\]
%dove $\Lambda^\sharp : \bigwedge^1\Gamma^*(T^\flat M_U,TM_U)$ è associato a $\Lambda$.\\
In particolare, $X_1 = X$.\\

L'applicazione $f \mapsto X_f $ è un morfismo di algebre di Lie:
\begin{proposition}
\label{P_MorfismoDiAlgebreDiLie}
Per ogni coppia $(f,g)$ di funzioni di $\mathfrak{A}(M)$, abbiamo~:
\[
\label{eq_MorfismoDiAlgebreDiLie}
[X_f,X_g] = X_{\{f,g\}}
\]
\end{proposition}

\begin{definition}
\label{D_MappaDiJacobi}
Siano $ \left( M_1,\Lambda_1,X_1 \right) $ e $ \left( M_2,\Lambda_2,X_2 \right) $ due varietà parziali di Jacobi.\\
Un'applicazione liscia è una \emph{mappa di Jacobi}\index{mappa!di Jacobi} se la mappa indotta 
$\varphi^\ast : C^\infty  \left( M_2 \right) \to C^\infty \left( M_1 \right) $ definita da 
$\varphi^\ast (f) = f \circ \varphi$ 
soddisfa le proprietà seguenti:
\[
\left\{
\begin{array}
[c]{c}
\varphi^\ast  \left( \mathfrak{A} \left( M_2 \right)  \right)
\subset
\mathfrak{A}  \left( M_1 \right) 
 \\
\forall (f,g) \in \mathfrak{A}  \left( M_2 \right)  ^2,\;
\{ \varphi^\ast (f),\varphi^\ast (g) \}_{M_1}
= \varphi^\ast \{f,g\}_{M_2} 
\end{array}
.\right.
\]
\end{definition}

%\begin{proposition}
%\label{P_CaratterizzazioneDiUnaMappaDiJacobi}
%Siano $ \left( M_1,\Lambda_1,X_1 \right) $ e $ \left( M_2,\Lambda_2,X_2 \right) $ due varietà parziali di Jacobi e $\varphi : M_1 \to M_2$ una applicazione liscia.\\ 
%Le proposizioni seguente sono equivalenti:
%\begin{enumerate}
%\item
%$\varphi$ è una mappa di Jacobi;
%\item
%$\forall f \in \mathfrak{A} \left( M_2 \right) ,\; \varphi_\ast  \left( \mathscr{H}_{f \circ \varphi} \right) 
%= \mathscr{H}_f$;
%\item
%$\varphi_\ast  \left( \Lambda_1 \right) 
%= \Lambda_2$  e  
%$\varphi_\ast  \left( X_1 \right) = X_2$. 
%\end{enumerate}
%\end{proposition}
%% \cite{Ryc2000}, Proposition 2.

\subsection{Esempi}
\label{___EsempiVarietaDiJacobiParziali}

\begin{example}
\label{Ex_VarietaDiJacobiDiDimensionFinita}
{\sf Varietà di Jacobi di dimensione finita.}\\
Per $M$ varietà di dimension finita, consideriamo $T^\flat M = T'M$ e l'algebra $\mathfrak{A}(M) = C^\infty(M)$. Il tensore $\Lambda$ è un campo tensoriale di tipo $(2,0)$, 
sezione del fibrato tensoriale $T_0^2 M$ (cf. \cite{AbTo11}, 3.2) e $X$ un campo vettoriale che soddisfano le condizioni di compatibilità
\[
\left\lbrace
\begin{array}{rcl}
[\Lambda,\Lambda]	&=&	2X \wedge \Lambda		\\
L_X \Lambda			&=&	0 
\end{array}
\right.
\] 
$(M,\Lambda,X)$ è une varietà di Jacobi parziale.
\end{example}

\begin{example}
\label{Ex_VarietaDiPoissonParzialiConvenienti}
{\sf Varietà di Poisson parziali convenienti.}\\
La nozione di varietà di Poisson parziale corrisponde a $X=0$ (cf. \cite{CaPe23}, 7.1).
\end{example}

\begin{example}
\label{Ex_LimiteDirettoDiVarietaDiJacobiDiDimensionFinite}
{\sf Limite diretto di varietà di Jacobi di dimensioni finite.}\\
Consideriamo la struttura di Jacobi sullo spazio vettoriale $\mathbb{R}^{2m+1}$ con le coordinate canoniche $ \left( x^0,x^1,\dots,x^{2m} \right) $ del esempio~\ref{Ex_StructureDeJacobiSurR2m+1}:
\[
X_m = \dfrac{\partial}{\partial x^0}
\textrm{~~~e~~~}
\Lambda_m = \displaystyle\sum_{i=1}^m
 \left( x^{m+i}\dfrac{\partial}{\partial x^0} - \dfrac{\partial}{\partial x^i} \right) 
 \wedge
\dfrac{\partial}{\partial x^{m+i}}.
\] 
Considerando l'iniezione naturale 
$\iota_{2m+1}^{2m+3} : \mathbb{R}^{2m+1} \to \mathbb{R}^{2m+3}$ che è una mappa di Jacobi, se definisce una successione $  \left( \mathbb{R}^{2m+1},\Lambda_m,X_m, \iota_{2m+1}^{2m+3} \right)  _{m \in \mathbb{N}}$ di varietà di Jacobi.\\
Il limite diretto (o limite induttivo)  $M=\displaystyle\underrightarrow{\lim}\mathbb{R}^{2m+1}$ può essere dotato di una struttura di spazio vettoriale conveniente. \\
Consideriamo la sottoalgebra delle funzioni cilindriche:
\[
C^\infty_{\operatorname{cyl}}(M)
=
\displaystyle\bigcup_{m \in \N}
\pi_m^\ast C^\infty \left( \R^{2m+1} \right)
\]
dove $\pi_m : M \to \R^{2m+1}$ è la proiezione canonica.\\
Una \emph{funzione cilindrica}\index{funzione cilindrica} è quindi della forma 
$f= f_m \circ \pi_m$ per un certo $m$ con $f_m \in C^\infty \left( \R^{2m+1} \right) $.\\
Per le funzioni cilindriche $f$ e $g$, esiste $N \in \N$ sufficientemente grande tale che:
\[
f = f_N \circ \pi_N \text{~~~e~~~} 
g = g_N \circ \pi_N
\]
dove $f_N$ e $g_N$ appartengono a 
$C^\infty \left( \R^{2N+1} \right) $.\\
Si definisce allora la parentesi di Jacobi $\{.,.\}$ su $C^\infty_{\operatorname{cyl}}(M)$ mediante:
\[
\{f,g\} = \{f_N,g_N\}_{ \left( \Lambda_N,E_N \right) } \circ \pi_N
\]
dove $\{.,.\}_{ \left( \Lambda_N,E_N \right) }$ 
è la parentesi di Jacobi standard su $\R^{2N+1}$.\\
Questa definizione è coerente, poiché le iniezioni sono mappe di Jacobi.\\
Si ottiene così una struttura parziale di Jacobi sul limite diretto $M$.
\end{example}

\begin{example}
\label{Ex_TrasformazioniConformidiVarietaDiJacobi}
{\sf Trasformazioni conformi di una varietà di Jacobi.}\\
Sia $\varphi$ un'applicazione di $\mathfrak{A}(M)$ che non si annulla mai. \\
La \emph{trasformazione conforme}\index{trasformazione conforme} di una varietà di Jacobi $(M,\Lambda,X)$ rispetto a $\varphi$ è definita dai tensori $\Lambda_\varphi$ e $X_{\varphi} $ seguenti: 
\begin{description}
\item[$\bullet$]
$\Lambda_\varphi = \varphi \Lambda$
\item[$\bullet$]
$X_{\varphi} 
= \varphi X + \Lambda^\sharp (d \varphi )$
\end{description}
A questa struttura di Jacobi è associata la parentesi (cf. \cite{Marl91}, 2.3, ex. 6):
\[
 \{ f,g \}_\varphi 
 = \dfrac{1}{\varphi} \{ \varphi f, \varphi g \}.
\]
\end{example}

\subsection{Varietà di Poisson omogenei parziali}
\label{VarietaDiPoissonOmogeneiParziali}

\begin{definition}
\label{D_VarietaDiPoissonOmogeneiParziale}
Si chiama \emph{varietà di Poisson omogenea parziale}\index{varietà!di Poisson omogenea parziale} 
una terna $(N,P,Z)$ costituita da una varietà di Poisson parziale $(N,P)$ e da un campo vettoriale $Z$, detto \emph{campo di omotetie}\index{campo!di omotetie}, che soddisfa la relazione 
\[
L_Z P = -P
\]
\end{definition}

A ogni struttura di Jacobi parziale è possibile associare una struttura di Poisson omogenea parziale.

\begin{proposition}
\label{P_VarietaDiPoissonOmogeneaParziale}
Sia $(M,\Lambda,E)$ una varietà di Jacobi parziale.\\
Poniamo $\hat{M} = M \times \R$, fibrato triviale in rette sopra $M$; indichiamo con $t$ la coordinata canonica sulla fibra $\R$ e con 
$Z = \dfrac{\partial}{\partial t}$ il campo vettoriale su $\hat{M}$ la cui proiezione su $\R$ è $1$ e la cui proiezione su $M$ è nulla. 
Sia $h : \hat{M} \to \R$ la funzione omogenea di grado $1$ rispetto a $Z$, definita da 
$h(x,t) = \exp(t)$.  
Definiamo sullo spazio $\hat{M}$ il tensore
\[
\hat{\Lambda} = \dfrac{1}{h}\, (\Lambda + Z \wedge E).
\]
Allora valgono le seguenti proprietà:
\begin{enumerate}
\item[1.] 
$ \left( \hat{M},\hat{\Lambda},Z \right) $ è una varietà di Poisson omogenea parziale.
\item[2.] 
La proiezione $\pi : \hat{M} \to M$ è un morfismo di Jacobi $h$-conforme.
\end{enumerate}
\end{proposition}

\begin{proof}
Sia $(M,\Lambda,E)$ una varietà di Jacobi parziale. Consideriamo la varietà 
$\hat{M} = M \times \R$, fibrato triviale in rette sopra $M$, insieme alla proiezione $\pi : \hat{M} \to M$ sul primo fattore.\\
A ogni funzione $g \in C^\infty(M)$ associamo la funzione 
$\hat{g} = h\,\pi^\ast g \in C^\infty(\hat{M})$, che risulta omogenea\footnote{Per $Z = \dfrac{\partial}{\partial t}$, abbiamo  $Z(\hat{g}) = \hat{g}$.} di grado $1$:
\[
\forall (x,t) \in \hat{M},\,
\hat{g}(x,t) = \e^t (g \circ \pi)(x,t) = \e^t g(x).
\]
La sua differenziale è data da
\[
d\hat{g} = h\big( \pi^\ast(dg) + (\pi^\ast g)\, dt \big).
\]

Per ogni aperto della forma $\hat{U} = U \times \R$ di $\hat{M}$ introduciamo l'insieme 
$\mathfrak{A}  \left( \hat{M} \right) $ delle funzioni lisce su $\hat{U}$ tali che, per ogni $\hat{x} \in \hat{U}$, ciascuna derivata di ordine superiore $d^k\hat{g}(\hat{x})$ soddisfi
\begin{equation}
\label{eq_dkhatf}
\forall (\hat{u}_2,\dots,\hat{u}_k) \in (T_{\hat{x}} \hat{M})^{k-1},\,
d^k_{\hat{x}} \hat{g}(\,\cdot\,,\hat{u}_2,\dots,\hat{u}_k) 
\in T^\flat_{\hat{x}} M \times T^\ast \R.
\end{equation}

\medskip
\noindent 1. La dimostrazione si basa in particolare sulle proprietà della parentesi di Schouten\footnote{Queste proprietà generalizzano i risultati ottenuti in \cite{FeMa14} e \cite{CFM21}. Per tensori $P$, $Q$, $R$, antisimmetrici rispettivamente $p$-, $q$-, $r$-controvarianti, valgono:
\[
[P,Q] = -(-1)^{(p-1)(q-1)} [Q,P], \qquad
[P,Q \wedge R] = [P,Q] \wedge R + (-1)^{(p-1)q} Q \wedge [P,R].
\]} 
indicate da \textbf{(CS2)} e \textbf{(CS3)} nel contesto parziale.

Siano $f \in \mathfrak{A}(M)$ e $T \in \Gamma^\ast_2(\mathfrak{A}(M))$.  
Da \textbf{(CS3)} segue:
\[
\begin{array}{rcl}
[fT,fT]
&=& [f \wedge T, f \wedge T]  \\
&=& [f \wedge T,f] \wedge T + f \wedge [f \wedge T,T].
\end{array}
\]
D'altra parte, per la proprietà \textbf{(CS2)}, si ha:
\[
[f \wedge T,f] 
= -(-1)^{(2-1)(0-1)} [f,f \wedge T]
= [f, f \wedge T].
\]
Applicando nuovamente \textbf{(CS3)}, otteniamo:
\[
[f,f \wedge T]
= [f,f] \wedge T + f \wedge [f,T]
= f \wedge [f,T].
\]
In modo analogo,
\[
\begin{array}{rcl}
[f \wedge T, T]
&=& -(-1)^{(2-1)(2-1)} [T, f \wedge T] \\
&=& [T,f] \wedge T + f \wedge [T,T].
\end{array}
\]
Poiché per \textbf{(CS2)} abbiamo
\[
[T,f] = -(-1)^{(2-1)(0-1)} [f,T] = [f,T],
\]
si ottiene infine:
\begin{equation}
[fT,fT] = f^2[T,T] + 2f[T,f] \wedge T.
\tag{1}
\end{equation}
Consideriamo ora $T = \Lambda + Z \wedge E$. Si ha:
\[
[T,T]
= [\Lambda,\Lambda] + 2[\Lambda,Z \wedge E] + [Z \wedge E, Z \wedge E].
\]
Poiché $(M,\Lambda,E)$ è una struttura di Jacobi, si ha:
\[
[\Lambda,\Lambda] = 2E \wedge \Lambda, \qquad [\Lambda,E] = 0,
\]
e le parentesi $[\Lambda,Z]$, $[Z,Z]$, $[E,E]$, $[Z,E]$ sono nulle, segue che
\[
[T,T] = 2E \wedge \Lambda.
\]
Inoltre, per ogni $f \in \mathfrak{A}(M)$, si ha $[\Lambda,f]=0$, e poiché
\[
[Z,f](t) = Z(f)(t) = -f(t),
\]
otteniamo
\[
[T,f] = -fE.
\]

Applicando la formula (1) al caso $f = \dfrac{1}{h}$ con $h(t) = e^t$,
\[
[\hat{\Lambda},\hat{\Lambda}] = 0,
\]
cioè $\hat{\Lambda}$ è un tensore di Poisson.

\smallskip

Per verificare che $\hat{\Lambda}$ è omogeneo di grado $-1$ rispetto à $Z$, ossia che vale 
\[
L_Z \hat{\Lambda} = -\hat{\Lambda},
\]
utilizziamo nuovamente \textbf{(CS3)}:
\[
L_Z \hat{\Lambda}
= [Z,fT]
= [Z,f] \wedge T + f[Z,T].
\]
Poiché $[Z,f] = -f$ e 
\[
[Z,T] = L_Z(\Lambda + Z \wedge E) = 0,
\]
il risultato è dimostrato.\\

\medskip
2. Per $(f,g) \in \mathfrak{A}(M)^2$,
\[
\{\hat{f},\hat{g}\}_{\hat{\Lambda}}
= \hat{\Lambda}(d\hat{f}, d\hat{g}).
\]
Un calcolo diretto mostra che
\[
\{\hat{f},\hat{g}\}_{\hat{\Lambda}}
= h\,\pi^\ast(\{f,g\}_\Lambda)
= \widehat{\{f,g\}_\Lambda}.
\]
Dunque l'applicazione
\[
f \longmapsto \hat{f} = h\,\pi^\ast f
\]
realizza un omomorfismo di parentesi che trasforma la parentesi di Jacobi $\{.,.\}_\Lambda$ su $M$ nella parentesi di Poisson $\{.,.\}_{\hat{\Lambda}}$ su $\hat{M}$ tramite la funzione $h$; ciò coincide con la definizione usuale di morfismo di Jacobi $h$-conforme.
\end{proof}

\subsection{Distribuzione caratteristica}
\label{___DistribuzioneCaratteristica}

Sia $(M,\Lambda,X)$ una varietà di Jacobi parziale conveniente.\\
%Una \emph{distribuzione}\index{distribuzione} sulla varietà $M$ è un sottoinsieme $\mathscr{D}$ del fibrato tangente $TM$ tale che $\mathscr{D}_x = \mathscr{D} \cap T_x M$ è un sottospazio di $T_x M$ per ogni $x \in M$.\\

La distribuzione $\mathscr{C}$ generata dai campi Hamiltoniani $X_f = \Lambda^\sharp df +fE$ 
%$\Lambda^\sharp  \left( T^\flat_x M \right) $ e  $\mathscr{H}_f(x)$ 
dove $f \in \mathfrak{A}(M)$ è chiamata \emph{distribuzione caratteristica}\index{distribuzione caratteristica} di la varietà di Jacobi parziale.\\

Poiché, in virtù della proposizione~\ref{P_MorfismoDiAlgebreDiLie}, per ogni coppia $(f,g)$ di funzioni di 
$\mathfrak{A}(M)$, vale la relazione
\[
\left[ X_f,X_g \right] = X_{\{f,g\}}
\]
si ottiene la seguente proposizione:   
\begin{proposition}
\label{D_DistribuzioneCaratteristicaInvolutiva}
La distribuzione caratteristica $\mathscr{C}$ è involutiva.
\end{proposition}

Nel quadro conveniente, l'integrabilità di una struttura di Poisson non è in generale garantita. In effetti, la difficoltà emerge già nel contesto delle varietà di Banach. Il lettore può trovare in \cite{PeCa19} delle condizioni sufficienti affinché un tale risultato perché questo risultato abbia luogo.
%Tuttavia, si dispongono delle seguenti condizioni sufficienti affinché tale risultato valga (cf. \cite{PeCa19}).

%\begin{theorem}
%\label{T_FoliazioneDiUnaVarietaDiJacobiParziale}
%Sia $ \left( T^\flat M, TM, P \right) $ una varietà di Banach-Poisson parziale per la quale il nucleo di $P$ sia scisso in ogni fibra $T^{\flat}_xM$ di $T^{\flat}M$ e tale che $P(T^{\flat}M)$ sia una distribuzione chiusa. Allora si ha:
%\begin{enumerate}
%\item
%$\mathcal{D} = P(T^{\flat}M)$ è integrabile e la foliazione definita da $\mathcal{D}$ è una foliazione debolmente simplettica, cioè tale che su ogni foglia $N$ si abbia una 2-forma chiusa e non degenere $\omega_N$ su $TN$ soddisfacente
%\[
%\omega_N(u,v) = \langle \alpha,, P(\beta)\rangle
%\]
%per ogni $u$ e $v$ in $T_xN$ e per ogni coppia $(\alpha,\beta)$ di elementi di $T_x^\flat M$ tali che $P(\alpha)=u$ e $P(\beta)=v$.
%\item
%Su ogni foglia massimale $N$ si considera la forma naturale debole $\omega_N$. Si indica con $\mathcal{E}_{\omega_N}(V)$ il sottoinsieme delle funzioni $f\in C^\infty(N)$ tali che $df$ appartenga a $\omega_N^\flat(TN)$. La restrizione $f_N$ di una qualunque $f \in \mathcal{E}(U)$ a $U\cap N$ appartiene a $\mathcal{E}_{\omega_N}(U\cap N)$ e, per ogni $(f,g) \in \mathcal{E}(U)^2$, vale
%\[
%\{f,g\}_{P_{|N}}
%=\{f_N,g_N\}_{\omega_N}.
%\]
%\end{enumerate}
%\end{theorem}

Il risultato seguente stabilisce condizioni sufficienti affinché la distribuzione caratteristica associata a una struttura di Jacobi parziale sia integrabile nel contesto di Banach.

\begin{theorem}
\label{T_CSIntegrabiliteDistributionCaracteristiqueJacobiVarieteDeBanach}
Sia $(M,\Lambda,E)$ una varietà di Banach dotata di una struttura di Jacobi parziale.\\
Se vale
\begin{itemize}
\item[\emph{\textbf{(CSI)}}]
l'immagine di $\Lambda^\sharp$ è un sottofibrato liscio, chiuso e scisso, cioè esiste un sottofibrato liscio, chiuso e supplementare\footnote{Ricordiamo che un sottospazio chiuso di uno spazio di Banach non ha necessariamente un supplemento (cf. \cite{Phi40}); ad esempio, nello spazio di Banach $\ell^\infty$ delle successioni reali limitate, il sottospazio chiuso $c_0$ delle successioni reali convergenti a $0$ non ha un supplemento.} $V$ tale che
\[
\forall x \in M,\; T_x M = \operatorname{im}\Lambda^\sharp_x\oplus V_x
\]
\end{itemize}
allora la distribuzione caratteristica della varietà di Jacobi parziale è integrabile.\\
Inoltre, le foglie della  foliazione associata sono le proiezioni su $M$ delle foglie simplettiche della struttura di Poisson omogenea associata
\[
\left( \hat{M}=M \times \R, 
\hat{\Lambda} 
= \dfrac{1}{h} (\Lambda + \partial_t \wedge E) \right)
\] 
dove $h(x,t)=\operatorname{exp}(t)$.
\end{theorem}

\begin{proof}
L'idea della dimostrazione dell'integrabilità completa, sviluppata nel caso di dimensione finita da Kirillov, consiste nel ridursi alla struttura di Poisson omogenea associata $ \left( \hat{M},\hat{\Lambda},Z \right) $ (cf. Proposizione~\ref{D_DistribuzioneCaratteristicaInvolutiva}). La distribuzione 
$\hat{\mathcal{C}} = \operatorname{Im}, \hat{\Lambda} \subset T\hat{M}$, associata al tensore di Poisson, è integrabile nel senso di Stefan-Sussmann e induce una foliazione 
$\hat{\mathcal{F}}$ le cui foglie sono sottovarietà simplettiche immerse (cf. \cite{Wei83} e \cite{Vai94}).\\
Poniamoci nel contesto di una varietà di Banach dotata di una struttura di Jacobi parziale
$(M,\Lambda,E)$ che soddisfa la condizione $\textbf{(CSI)}$.\\
Consideriamo inoltre la struttura di Poisson omogenea associata
$ \left( \hat{M},\hat{\Lambda},Z  \right) $ dove $\hat{M} = M \times \R$ e 
$\hat{\Lambda} 
= \dfrac{1}{h} (\Lambda + \partial_t \wedge E)$ con $h(x,t) = \operatorname{exp}(t)$. \\
D'altra parte, indichiamo con $p: \hat{M} \to M$ la proiezione canonica associata.\\
Per ogni punto $(x,t) \in \hat{M}$ e per ogni funzione $f \in \mathfrak{A}(M)$, si ha
\[
\hat{\Lambda}^\sharp_{(x,t)}
 \left( p^\ast(df) + (f \circ p) dt \right) 
 =
\e^{-t}
\left(
\Lambda^\sharp_x(df) + f(x)E_x -df_x (E_x) \partial t 
\right)
\]
che può essere riscritto usando il campo Hamiltoniano $X_f=\Lambda^\sharp$ come
\[
\hat{\Lambda}^\sharp_{(x,t)}
 \left( p^\ast(df) + (f \circ p) dt \right) 
 =
\e^{-t}
\left(
X_f -df_x (E_x) \partial t 
\right)
\]
La distribuzione caratteristica $\hat{\mathcal{C}}$ della struttura di Poisson omogenea $\hat{\Lambda}$ è allora definita, per ogni punto $(x,t) \in \hat{M}$, da
\[
\hat{\mathcal{C}}(x,t) = \operatorname{im} \left( \hat{\Lambda}^\sharp _{(x,t)} \right)  
=
\operatorname{span} 
\left\lbrace 
X_f (x) - df_x \left( E_x \right) \partial t
\right\rbrace
\]
La proiezione di questa distribuzione tramite $d p$ coincide esattamente con la distribuzione $\mathcal{C}$.\\

D'altra parte, la condizione $\textbf{(CSI)}$ garantisce un'analoga proprietà per l'immagine di $\hat{\Lambda}^\sharp$: essa assicura infatti l'esistenza di un fibrato liscio $\hat{V}$, supplementare chiuso del fibrato liscio chiuso $\operatorname{im} \hat{\Lambda}^\sharp$.\
Poiché inoltre la distribuzione di Poisson $\hat{C}$ è involutiva, il teorema di Frobenius per varietà di Banach assicura l'integrabilità della distribuzione caratteristica $\hat{C}$ (cf. \cite{Omo97}) e dunque l'esistenza di una foliazione  $\hat{\mathcal{F}}$ le cui foglie sono sottovarietà simplettiche immerse.\\

Infine, poiché la proiezione $p$ è una submersione liscia suriettiva e trasversale alla foliazione $\hat{\mathcal{F}}$, le foglie caratteristiche della foliazione $\mathcal{F}$ associata alla struttura di Jacobi su $M$ si ottengono come proiezioni delle foglie della foliazione $\hat{\mathcal{F}}$.
\end{proof}

\begin{remark}
\label{R_TipiDiFoglieDiversi}
Questa foliazione è costituita da due tipi di foglie.\\
Se il campo $E$ è contenuto nell'immagine di $\Lambda^\sharp$, la foglia $F$ può essere dotata di una struttura di varietà localmente conformemente simplettica. Nel caso contrario, la foglia $F$ potrebbe essere dotata di una struttura che generalizza, nel contesto delle varietà di Banach, la nozione di varietà di contatto. La condizione $\theta \wedge (d\theta)^m$ che è una forma di volume, non ha più senso in questo contesto di dimensione infinita;  potrebbe essere sostituita dall'esistenza di una $1$-forma $\theta$ tale che, su $F$, si abbia $\theta(E) = 1$, $\ker \theta = \operatorname{im}\Lambda^\sharp$ e $d\theta_x|_{\ker \theta_x}$ non degenere.
\end{remark}

D'altra parte, nel contesto conveniente, esistono condizioni sufficienti per l'integrabilità di una distribuzione di rango finito, localmente generata da particolari tipi di campi Hamiltoniani.

\begin{theorem}
\label{T_CondizioniSufficientiDiIntegrabilitaDistribuzioneHamiltonianaRangoFinitoSuVarietaConveniente}
Sia $M$ una varietà conveniente e sia $F$ un sottofibrato di dimensione finita $n$ del fibrato tangente cinematico $TM$.\\
Se, per ogni $x \in M$, esiste un intorno aperto $U$ di $x$ e $n$ campi vettoriali Hamiltoniani locali $X_{f_1},\dots,X_{f_n}$ (con $f_i \in \mathfrak{A}(M)$), aventi flussi locali\footnote{Nel contesto conveniente, un campo vettoriale cinematico non ha necessariamente un flusso locale. Infatti, al di là del contesto degli spazi di Banach, i risultati classici sull'esistenza e unicità delle soluzioni delle equazioni differenziali, derivanti da teoremi del punto fisso, non si applicano più necessariamente (cf. \cite{KrMi97}, 32.12).}
$\operatorname{Fl}_t^{X_{f_1}},\dots,\operatorname{Fl}_t^{X_{f_n}}$, 
allora la distribuzione $F$ è integrabile.
\end{theorem}

\begin{proof}
Si applica \cite{Tei01}, Theorem~2, alla distribuzione involutiva generata localmente dai campi Hamiltoniani $A_i=X_{f_i}$ ($i \in \{1,\dots,n\}$), ciascuno dei quali possiede un flusso locale.
\end{proof}

%
%\subsection{Riduzione d'una varietà di Jacobi parziale}
%
%
%%\cite{DLM91}, Proposition 3.6.
%\begin{proposition}
%\label{P_RiduzioneDUnaVarietaDiJacobiParziale}
%Siano $(M,\Lambda,X)$ una varietà di Jacobi parziale conveniente e $\iota : N \to M$ una sottovarietà di $M$ dove $\iota$ è un morfismo di Jacobi.\\
%Se la restrizione $\mathscr{C}_N$ della distribuzione caratteristica $\mathscr{C}$ è contenuta in $TN$, allora esiste sulla sottovarietà una struttura di Jacobi.
%\end{proposition}
%
%\begin{proof}
%Sia $\iota : \check{M} \to M$ una sotto varietà della varietà $(M,\Lambda,X)$.\\
%Se la restrizione $\mathscr{C}_N$ della distribuzione caratteristica $\mathscr{C}$ è contenuta in $TN$, si ha
%\[
%\iota_\ast  \left( \check{\Lambda} \right) 
%= \Lambda  \textrm{  e  }  
%\iota_\ast  \left( \check{X} \right) = X.
%\] 
%Otteniamo
%\[
%\iota_\ast  \left( \mathscr{H}_{f \circ \iota} \right) 
%= \mathscr{H}_f
%\]
%%!..
%%Il campo di vettori Hamiltoniano $\mathscr{H}  \left( \varphi^\ast (f) \right) $  si proietta sul campo di vettori Hamiltoniano $\mathscr{H}(f)$.
%% \cite{DLM91}, 1.6
%\end{proof}
%% \cite{Marl2000b}, Proposition 3.
%% \cite{Marl2000b}, Théorème 3.9
%% $T_xM = T_x N \oplus \Lambda^\sharp  \left( T_x N^0 \right) 
%
%%\subsection{Varietà di Poisson omogenea parziale}

%Lie algebroid structure of the jet bundle $J^1(M,\mathbb{R})$ 
%Y. Kerbrat, Z. Souici-Benhammadi,
% Variétés de Jacobi et groupoïdes de
%contact. C. R. Acad. Sci. Paris, S ?er. I, 317 (1993) 81--86.

\section{Sviluppi ulteriori}
\label{SviluppiUlteriori}

Vengono qui suggerite alcune linee di ricerca future relative alla teoria delle varietà di Jacobi parziali convenienti.
\begin{enumerate}
\item
Ci si può interessare ai problemi relativi alla restrizione di una struttura di Jacobi parziale su una varietà conveniente $M$ a una sotto-varietà $N$ di $M$, come è stato fatto in dimensione finita da C.-M. Marle in \cite{Marl2000}: si cercano allora condizioni sufficienti affinché la restrizione di tale struttura a $N$ erediti una struttura analoga. In questo modo, si generalizza la nozione di sotto-varietà di Poisson di A.~Weinstein (cf. \cite{Wei83}) nonché le strutture di Poisson sullo spazio delle fasi di un sistema meccanico con vincoli cinematici di Van der Schaft (cf. \cite{VdSMa94}).
\item
La nozione di fibrato di Jacobi, intesa come generalizzazione del concetto di varietà di Jacobi, è stata introdotta, in dimensione finita, da C.-M. Marle in \cite{Marl91}, dove si dimostra che lo spazio totale di un tale fibrato è dotato di una struttura di varietà di Poisson omogenea. Si potrebbe quindi, innanzitutto, definire una nozione di fibrato di Jacobi parziale su una varietà conveniente e verificare se questo risultato possa essere esteso a questo contesto.
\item
In \cite{CaPe23}, 7.2, viene introdotta la nozione di algebroide di Lie parziale e viene messo in evidenza un legame con le strutture di Poisson parziali. Sorge quindi il problema di capire se tale legame possa essere esteso alle strutture di algebroidi di Jacobi parziali e alle varietà di Jacobi parziali.
\item
I limiti diretti di successioni ascendenti di strutture di dimensione finita forniscono numerosi esempi interessanti di strutture convenienti in algebra:
$\R^{\infty}=\underrightarrow{\lim}\R^n$  (\cite{Spa14}, esempio~3.1),
$\mathbb{S}^{\infty}=\underrightarrow{\lim}\mathbb{S}^n$ (\cite{KrMi97}, 47.2), 
$\operatorname{GL}(\infty,\R)=\underrightarrow{\lim}\operatorname{GL}(n,\R)$ (\cite{KrMi97}, 47.8), etc.\
D'altra parte, i limiti diretti di alcune successioni crescenti di strutture parziali di Poisson, di Nambu-Poisson o anche di Dirac, definite su varietà di Banach, forniscono esempi di strutture parziali convenienti dello stesso tipo (cf. rispettivamente \cite{CaPe23}, \cite{PeCa24a} e \cite{PeCa24b}).\\
Risulta quindi interessante individuare condizioni sufficienti sulle successioni di strutture parziali di Jacobi tali che il loro limite diretto sia dotato di una struttura parziale di Jacobi conveniente.
\end{enumerate}

%%%%%%%%%%%% Appendices %%%%%%%%%%%%%%%%%%%%
%\renewcommand{\thesection}{\Alph{section}}
%\setcounter{theorem}{0}
%\setcounter{section}{0}

%https://topkorae.com/wiki/it/Convenient_vector_space
%https://it.upwiki.one/wiki/Convenient_vector_space

\end{document}